\documentclass[english]{article}
\usepackage[T1]{fontenc}
\usepackage[latin9]{inputenc}
\usepackage{geometry}
\usepackage{prettyref}
\usepackage{units}
\usepackage{amsthm}
\usepackage{amsmath}
\usepackage{amssymb}
\usepackage{fixltx2e}

\makeatletter

\newcommand{\noun}[1]{\textsc{#1}}

\numberwithin{figure}{section}
\numberwithin{table}{section}
\numberwithin{equation}{section}
  \theoremstyle{plain}
  \newtheorem*{thm*}{\protect\theoremname}
\theoremstyle{plain}
\newtheorem{thm}{\protect\theoremname}[section]
  \theoremstyle{definition}
  \newtheorem{defn}[thm]{\protect\definitionname}
  \theoremstyle{definition}
  \newtheorem*{example*}{\protect\examplename}
  \theoremstyle{plain}
  \newtheorem{lem}[thm]{\protect\lemmaname}
  \theoremstyle{remark}
  \newtheorem*{acknowledgement*}{\protect\acknowledgementname}
  \theoremstyle{plain}
  \newtheorem{prop}[thm]{\protect\propositionname}
\newcommand{\lyxaddress}[1]{
\par {\raggedright #1
\vspace{1.4em}
\noindent\par}
}

\usepackage{graphicx}
\usepackage{setspace}
\usepackage[hyphens]{url}
\usepackage[perpage,symbol]{footmisc}
\usepackage{enumitem}
\usepackage{multicol}
\usepackage{etoolbox}
\usepackage{relsize}

\DefineFNsymbols*{lamportnostar}[math]{{(\dagger)}{(\ddagger)}{(\S)}{(\P)}{(\|)}{(\dagger\dagger)}{(\ddagger\ddagger)}}
\setfnsymbol{lamportnostar}

\DeclareMathOperator{\im}{im}
\DeclareMathOperator{\sgn}{sgn}
\DeclareMathOperator{\Spec}{Spec}
\DeclareMathOperator{\tr}{tr}
\DeclareMathOperator{\avg}{avg}

\DeclareMathOperator{\overlap}{overlap}

\newrefformat{prop}{Proposition~\ref{#1}}
\newrefformat{lem}{Lemma~\ref{#1}}
\newrefformat{sec}{§\ref{#1}}
\newrefformat{sub}{§\ref{#1}}

\setitemize[0]{leftmargin=10pt,itemindent=0pt}

\allowdisplaybreaks[1]

\renewenvironment{thebibliography}[1]{%
    \begin{oldthebibliography}{#1}%
      \setlength{\parskip}{0ex}%
      \setlength{\itemsep}{0ex}%
      \begin{small}
  }%
  {%
    \end{small}
    \end{oldthebibliography}%
  }

\makeatother

\usepackage{babel}
  \providecommand{\acknowledgementname}{Acknowledgement}
  \providecommand{\definitionname}{Definition}
  \providecommand{\examplename}{Example}
  \providecommand{\lemmaname}{Lemma}
  \providecommand{\propositionname}{Proposition}
  \providecommand{\theoremname}{Theorem}
\providecommand{\theoremname}{Theorem}

\begin{document}

\title{Mixing in high-dimensional expanders}

\author{Ori Parzanchevski %
\thanks{Supported by The Fund for Math at the Institute for Advanced Study.%
}}

\maketitle

\begin{abstract}
We prove a generalization of the Expander Mixing Lemma for arbitrary
(finite) simplicial complexes. The original lemma states that concentration
of the Laplace spectrum of a graph implies combinatorial expansion
(which is also referred to as \emph{mixing}, or \emph{quasi-randomness}).
Recently, an analogue of this Lemma was proved for simplicial complexes
of arbitrary dimension, provided that the skeleton of the complex
is complete. More precisely, it was shown that a concentrated spectrum
of the simplicial Hodge Laplacian implies a similar type of expansion
as in graphs. In this paper we remove the assumption of a complete
skeleton, showing that concentration of the Laplace spectra in all
dimensions implies combinatorial expansion in any complex. As applications
we show that spectral concentration implies Gromov's geometric overlap
property, and can be used to bound the chromatic number of a complex
\end{abstract}

\section{Introduction}

The \emph{spectral gap} of a finite graph $G=\left(V,E\right)$ is
the smallest nontrivial eigenvalue of its Laplacian operator. The
\emph{discrete Cheeger inequalities} \cite{Dod84,Tan84,AM85,Alo86}
relate the spectral gap to expansion in the graph: If the spectral
gap is large, then for any partition $V=A\cup B$ there is a large
number of edges connecting a vertex in $A$ with a vertex in $B$.
Nevertheless, a large spectral gap does not suffice to control the
number of edges between \emph{any }two sets of vertices. For example,
there exist ``bipartite expanders'' (see e.g.\ \cite{LPS88,marcus2013interlacing}):
graphs with a large spectral gap, which are bipartite, so that there
are $A,B\subseteq V$ of size $\left|A\right|=\left|B\right|=\frac{\left|V\right|}{4}$
with no edges between them. The \emph{Expander Mixing Lemma }\cite{friedman1987expanding,AC88,beigel1993fault}
remedies this inconvenience, using not only the spectral gap but also
the maximal eigenvalue of the Laplacian:
\begin{thm*}[Expander Mixing Lemma, \cite{friedman1987expanding,AC88,beigel1993fault}]
Let $G=\left(V,E\right)$ be a graph on $n$ vertices. If the nontrivial
spectrum of the Laplacian is contained within $\left[k\left(1-\varepsilon\right),k\left(1+\varepsilon\right)\right]$,
then for any two sets of vertices $A,B$ one has
\[
\left|\left|E\left(A,B\right)\right|-\frac{k}{n}\left|A\right|\left|B\right|\right|\leq\varepsilon k\sqrt{\left|A\right|\left|B\right|},
\]
where $E\left(A,B\right)$ are the edges with one endpoint in $A$
and the other in $B$.
\end{thm*}
If $k$ is the average degree of a vertex in $G$, then $\frac{k}{n}\left|A\right|\left|B\right|$
is about the expected size of $\left|E\left(A,B\right)\right|$ (the
exact value is $\frac{k}{n-1}\left|A\right|\left|B\right|$). Thus,
the Lemma means that a concentrated spectrum indicates a quasi-random
behavior. In light of the Expander Mixing Lemma, we call a graph whose
nontrivial Laplace spectrum is contained in $\left[k\left(1-\varepsilon\right),k\left(1+\varepsilon\right)\right]$
a \emph{$\left(k,\varepsilon\right)$-expander}%
\footnote{In \cite{Tao11} this is referred to as a ``\emph{two-sided $\left(k,\varepsilon\right)$-expander}'',
as the spectrum is bounded on both sides.%
}. 

In \cite{parzanchevski2012isoperimetric} a generalization of the
Expander Mixing Lemma was proved for simplicial complexes of arbitrary
dimension, assuming that they have a complete skeleton%
\footnote{A $d$-dimensional complex is said to have a complete skeleton if
every cell of dimension smaller than $d$ is in the complex. For example,
a triangle complex with a complete underlying graph. Such complexes
are sometimes called \emph{hypergraphs}.%
}. The Laplace operator which is studied there and in the current paper
originates in Eckmann's work \cite{Eck44}. It is a natural analogue
of the Hodge Laplace operator in Riemmanian geometry, and it was studied
in several prominent works \cite{Gar73,Zuk96,friedman1998computing,kook2000combinatorial,ABM05,duval2009simplicial},
sometimes under the name \emph{combinatorial Laplacian}. More precisely,
a complex of dimension $d$ has $d$ Laplace operators (defined here
in \prettyref{sec:Simplicial-Hodge-theory}), with the $j$-th one
acting on the cells of dimension $j$ ($0\leq j<d$). We say that
$X$ is a \emph{$\left(j,k,\varepsilon\right)$-expander} if $\varepsilon<1$,
and the nontrivial spectrum of the $j$-th Laplacian is contained
in $\left[k\left(1-\varepsilon\right),k\left(1+\varepsilon\right)\right]$
(see \prettyref{sub:Spectrum} for the precise definition). 
\begin{thm*}[\cite{parzanchevski2012isoperimetric}]
Let $X$ be a $d$-complex on $n$ vertices with a complete skeleton,
which is a $\left(d-1,k,\varepsilon\right)$-expander. For any disjoint
$A_{0},\ldots,A_{d}\subseteq V$,
\[
\left|\left|F\left(A_{0},\ldots,A_{d}\right)\right|-\frac{k}{n}\left|A_{0}\right|\cdot\ldots\cdot\left|A_{d}\right|\right|\leq\varepsilon k\left(\left|A_{0}\right|\cdot\ldots\cdot\left|A_{d}\right|\right)^{\frac{d}{d+1}}
\]
where $F\left(A_{0},\ldots,A_{d}\right)$ is the set of $d$-cells
with one vertex in each $A_{i}$.
\end{thm*}
In this paper we prove a mixing lemma for arbitrary (finite) complexes.
Our main result is the following (this is a special case of \prettyref{prop:from_j_to_l}):
\begin{thm}
\label{thm:Mixing-Lemma}If a $d$-dimensional complex $X$ is a $\left(j,k_{j},\varepsilon_{j}\right)$-expander
for every $0\leq j\leq d-1$, and $A_{0},\ldots,A_{d}$ are disjoint
sets of vertices in $X$ then
\[
\left|\left|F\left(A_{0},\ldots,A_{d}\right)\right|-\frac{k_{0}\ldots k_{d-1}}{n^{d}}\left|A_{0}\right|\cdot\ldots\cdot\left|A_{d}\right|\right|\leq c_{d}k_{0}\ldots k_{d-1}\left(\varepsilon_{0}+\ldots+\varepsilon_{d-1}\right)\max\left|A_{i}\right|,
\]
 where $c_{d}$ depends only on $d$.
\end{thm}
In order to understand $F\left(A_{0},\ldots,A_{d}\right)$ in the
case of general complexes, we study a wider counting problem:
\begin{defn}
Given disjoint sets $A_{0},\ldots,A_{\ell}\subseteq V$, and $j\leq\ell$,
a \emph{$j$-gallery in $A_{0},\ldots,A_{\ell}$} is a sequence of
$j$-cells $\sigma_{0},\ldots,\sigma_{\ell-j}\in X^{j}$, such that
$\sigma_{i}$ is in $F\left(A_{i},\ldots,A_{i+j}\right)$, and $\sigma_{i}$
and $\sigma_{i+1}$ intersect in a $\left(j-1\right)$-cell (which
must lie in $F\left(A_{i+1},\ldots,A_{i+j}\right)$). We denote the
set of $j$-galleries in $A_{0},\ldots,A_{\ell}$ by $F^{j}\left(A_{0},\ldots,A_{\ell}\right)$. \end{defn}
\begin{example*}
$ $
\begin{enumerate}
\item An $\ell$-gallery in $A_{0},\ldots,A_{\ell}$ is just a single $\ell$-cell,
so that $F^{\ell}\left(A_{0},\ldots,A_{\ell}\right)=F\left(A_{0},\ldots,A_{\ell}\right)$.
\item A $0$-gallery is any sequence of vertices, so that $F^{0}\left(A_{0},\ldots,A_{\ell}\right)=A_{0}\times\ldots\times A_{\ell}$.
\item $F^{2}\left(A,B,C,D,E\right)$ is the number of triplets of triangles
$t_{1}\in F\left(A,B,C\right)$, $t_{2}\in F\left(B,C,D\right)$,
$t_{3}\in F\left(C,D,E\right)$ such that the boundaries of $t_{1}$
and $t_{2}$ share a common edge (necessarily in $F\left(B,C\right)$),
and likewise for $t_{2}$ and $t_{3}$.
\end{enumerate}
\end{example*}
The heart of our analysis is the following lemma, which estimates
the size of $F^{j+1}\left(A_{0},\ldots,A_{\ell}\right)$ in terms
of that of $F^{j}\left(A_{0},\ldots,A_{\ell}\right)$. Repeatedly
applying this lemma allows us to estimate $\left|F\left(A_{0},\ldots,A_{d}\right)\right|=\left|F^{d}\left(A_{0},\ldots,A_{d}\right)\right|$
in terms of $\left|F^{0}\left(A_{0},\ldots,A_{d}\right)\right|=\left|A_{0}\right|\cdot\ldots\cdot\left|A_{d}\right|$.
\begin{lem}[Descent Lemma]
\label{lem:Descent}Let $A_{0},\ldots,A_{\ell}$ be disjoint sets
of vertices in $X$.%
\footnote{In fact, it suffices that each $j+1$ tuple $A_{i},A_{i+1},\ldots,A_{i+j+1}$
consist of disjoint set.%
} If $X$ is an $\left(i,k_{i},\varepsilon_{i}\right)$-expander for
$i=j-1,i=j$, then 
\begin{multline*}
\left|\left|F^{j+1}\left(A_{0},\ldots,A_{\ell}\right)\right|-\left(\frac{k_{j}}{k_{j-1}}\right)^{\ell-j}\left|F^{j}\left(A_{0},\ldots,A_{\ell}\right)\right|\right|\\
\leq\left(\ell-j\right)k_{j}^{\ell-j}\left(\varepsilon_{j}+\varepsilon_{j-1}\right)\sqrt{\left|F\left(A_{0},\ldots,A_{j}\right)\right|\left|F\left(A_{\ell-j},\ldots,A_{\ell}\right)\right|}.
\end{multline*}

\end{lem}
The proofs of this lemma and of the mixing lemma it implies (\prettyref{thm:Mixing-Lemma})
appear in \prettyref{sec:The-main-theorems}, after giving the required
definitions in \prettyref{sec:Simplicial-Hodge-theory}. In \prettyref{sec:Applications}
we demonstrate applications of the mixing lemma, showing that spectral
expanders form geometric expanders (in the sense of Gromov, see \cite{Gro10,MW11}),
and have large chromatic numbers. We also present the idea of \emph{ideal
expanders} in this section, and list some open problems in \prettyref{sec:Questions}.
\begin{acknowledgement*}
I would like to thank Konstantin Golubev, Alex Lubotzky, Ron Rosenthal
and Doron Puder for many valuable discussions.
\end{acknowledgement*}

\section{\label{sec:Simplicial-Hodge-theory}Simplicial Hodge theory}

We describe here briefly the notions we shall need from the so-called
simplicial Hodge theory, originating in \cite{Eck44}. For a more
detailed summary we refer the reader to \cite[§2]{parzanchevski2012isoperimetric}.

Let $X$ be a $d$-dimensional simplicial complex on $n$ vertices
$V$. For $-1\leq j\leq d$ we denote by $X^{j}$ the set of $j$-cells
in $X$ (cells of size $j+1$), and by $X_{\pm}^{j}$ the set of oriented
$j$-cells, i.e.\ ordered cells up to an even permutation. A \emph{$j$-form}
on $X$ is an antisymmetric function on oriented $j$-cells:
\[
\Omega^{j}=\Omega^{j}\left(X\right)=\left\{ f:X_{\pm}^{j}\rightarrow\mathbb{R}\,\middle|\, f\left(\overline{\sigma}\right)=-f\left(\sigma\right)\;\forall\sigma\in X_{\pm}^{j}\right\} ,
\]
where $\overline{\sigma}$ is $\sigma$ endowed with the opposite
orientation. In dimensions $0$ and $-1$ there is only one orientation,
and so $\Omega^{0}=\mathbb{R}^{V}$ and $\Omega^{-1}=\mathbb{R}^{\left\{ \varnothing\right\} }\cong\mathbb{R}$.
The \emph{$j^{\mathrm{th}}$ boundary} \emph{operator }$\partial_{j}:\Omega^{j}\rightarrow\Omega^{j-1}$
is defined by $\left(\partial_{j}f\right)\left(\sigma\right)=\sum_{v\cup\sigma\in X^{j}}f\left(v\sigma\right)$.
The sequence $\Omega^{-1}\overset{{\scriptscriptstyle \partial_{0}}}{\longleftarrow}\Omega^{0}\overset{{\scriptscriptstyle \partial_{1}}}{\longleftarrow}\ldots$
is a chain complex, i.e.\ $B_{j}\overset{{\scriptscriptstyle def}}{=}\im\partial_{j+1}\subseteq\ker\partial_{j}\overset{{\scriptscriptstyle def}}{=}Z_{j}$,
and $H^{j}=\nicefrac{Z_{j}}{B_{j}}$ is the $j^{\mathrm{th}}$ (real,
reduced) homology group of $X$. We endow each $\Omega^{j}$ with
the inner product $\left\langle f,g\right\rangle =\sum_{\sigma\in X^{j}}f\left(\sigma\right)g\left(\sigma\right)$,
which gives rise to a dual \emph{coboundary operator} $\delta_{j}=\partial_{j}^{*}:\Omega^{j-1}\rightarrow\Omega^{j}$.
The real \emph{cohomology }of $X$ is $H^{j}=\nicefrac{Z^{j}}{B^{j}}$,
where $B^{j}\overset{{\scriptscriptstyle def}}{=}\im\delta_{j}\subseteq\ker\delta_{j+1}\overset{{\scriptscriptstyle def}}{=}Z^{j}$,
and by the fundamental theorem of linear algebra one has $B_{j}^{\bot}=Z^{j}$
and $Z_{j}^{\bot}=B^{j}$.

Simplicial Hodge theory, originating in \cite{Eck44}, studies the
\emph{upper, lower} and \emph{full Laplacians}: $\Delta_{j}^{+}=\partial_{j+1}\delta_{j+1}$,
$\Delta_{j}^{-}=\delta_{j}\partial_{j}$, and $\Delta_{j}=\Delta_{j}^{+}+\Delta_{j}^{-}$,
respectively. All of the Laplacians are self-adjoint and decompose
with respect to the orthogonal decompositions $\Omega^{j}=B^{j}\oplus Z_{j}=B_{j}\oplus Z^{j}$,
and the following properties are simple exercises:
\begin{gather*}
\begin{aligned}Z^{j} & =\ker\Delta_{j}^{+} & B_{j} & =\im\Delta_{j}^{+} & Z_{j} & =\ker\Delta_{j}^{-} & B^{j} & =\im\Delta_{j}^{-}\end{aligned}
\\
Z^{j}\cap Z_{j}=\left(B_{j}\oplus B^{j}\right)^{\bot}=\ker\Delta_{j}\cong H_{j}\cong H^{j}\quad\mbox{(Discrete Hodge Theorem).}
\end{gather*}
The dimension of $\ker\Delta_{j}\cong H_{j}\cong H^{j}$ is the $j^{\mathrm{th}}$
(reduced) \emph{Betti number} of $X$, denoted by $\beta_{j}$.

\smallskip{}

The combinatorial meaning of the Laplacians is better understood via
the following adjacency relations on oriented cells:
\begin{enumerate}
\item For two oriented $j$-cells $\sigma,\sigma'$, we denote $\sigma\pitchfork\sigma'$
if $\sigma$ and $\sigma'$ intersect in a common $\left(j-1\right)$-cell
and induce the same orientation on it; for edges this means that they
have a common origin or a common endpoint, and for vertices $v\pitchfork v'$
holds whenever $v\neq v'$. 
\item We denote $\sigma\sim\sigma'$ if: $\sigma\pitchfork\sigma'$, and
in addition the $\left(j+1\right)$-cell $\sigma\cup\sigma'$ is in
$X$. For vertices this is the common relation of neighbors in a graph%
\footnote{This adjacency relation can be used to define a stochastic process
on $j$-cells whose properties relate to the homology of the complex
- see \cite{PR12}.%
}.
\end{enumerate}
Using these relations, the Laplacians can be expressed as follows
(here the degree of a $j$-cell is the number of $\left(j+1\right)$-cells
in which it is contained):
\begin{align*}
\left(\Delta_{j}^{+}\varphi\right)\left(\sigma\right) & =\deg\left(\sigma\right)\varphi\left(\sigma\right)-\sum_{\sigma'\sim\sigma}\varphi\left(\sigma'\right)\\
\left(\Delta_{j}^{-}\varphi\right)\left(\sigma\right) & =\left(j+1\right)\varphi\left(\sigma\right)+\sum_{\sigma'\pitchfork\sigma}\varphi\left(\sigma'\right)\\
\left(\Delta_{j}\varphi\right)\left(\sigma\right) & =\left(\deg\sigma+j+1\right)\varphi\left(\sigma\right)+\sum_{{\sigma'\pitchfork\sigma\atop \sigma'\nsim\sigma}}\varphi\left(\sigma'\right)
\end{align*}
We shall also define adjacency operators on $\Omega^{j}$ which correspond
to the $\sim$ and $\pitchfork$ relations:
\[
\left(\mathcal{A}_{j}^{\sim}\varphi\right)\left(\sigma\right)=\sum_{\sigma'\sim\sigma}\varphi\left(\sigma'\right),\qquad\left(\mathcal{A}_{j}^{\pitchfork}\varphi\right)\left(\sigma\right)=\sum_{\sigma'\pitchfork\sigma}\varphi\left(\sigma'\right),
\]
so that $\Delta_{j}^{-}=\left(j+1\right)\cdot I+\mathcal{A}_{j}^{\pitchfork}$
and $\Delta_{j}^{+}=D_{j}-\mathcal{A}_{j}^{\sim}$, where $D_{j}$
is the degree operator $\left(D_{j}f\right)\left(\sigma\right)=\deg\left(\sigma\right)f\left(\sigma\right)$.

\subsection{\label{sub:Spectrum}Spectrum}

The spectra we are primarily interested in are those of $\Delta_{j}^{+}$
for $0\leq j\leq d-1$. Since $\left(\Omega^{j},\delta_{j}\right)$
is a co-chain complex, $B^{j}=\im\delta_{j}$ must be contained in
the kernel of $\Delta_{j}^{+}=\partial_{j+1}\delta_{j+1}$, and the
zero eigenvalues which correspond to forms in $B^{j}$ are considered
to be the \emph{trivial spectrum} of $\Delta_{j}^{+}$. As $\left(B^{j}\right)^{\bot}=Z_{j}$,
we call $\Spec\Delta_{j}^{+}\big|_{Z_{j}}$ the \emph{nontrivial spectrum}
of $\Delta_{j}^{+}$. Note that zero is a nontrivial eigenvalue of
$\Delta_{j}^{+}$ precisely when $Z_{j}\cap Z^{j}\neq0$, i.e.\ $\beta_{j}\neq0$.
For example, the constant functions on $V$ form the trivial eigenfunctions
of $\Delta_{0}^{+}$. The nontrivial spectrum of $\Delta_{j}^{+}$
corresponds to $Z_{0}$, which are the functions whose sum on all
vertices vanish, and zero is a nontrivial eigenvalue of $\Delta_{0}^{+}$
iff the complex is disconnected.

As hinted in the introduction, we say that $X$ is a \emph{$\left(j,k,\varepsilon\right)$-expander}
if $\varepsilon<1$ and $\Spec\Delta_{j}^{+}\big|_{Z_{j}}\subseteq\left[k\left(1-\varepsilon\right),k\left(1+\varepsilon\right)\right]$.
Given $\overline{k}=\left(k_{0},\ldots,k_{d-1}\right)$ and $\overline{\varepsilon}=\left(\varepsilon_{0},\ldots,\varepsilon_{d-1}\right)$,
we say that $X$ is a \emph{$\left(\overline{k},\overline{\varepsilon}\right)$-expander}
if it is a $\left(j,k_{j},\varepsilon_{j}\right)$-expander for all
$j$. The restriction $\varepsilon_{j}<1$ ensures that $X$ has trivial
$j$-th homology, i.e.\ $\beta_{j}=0$. While some of our results
hold for general $\varepsilon$ (e.g.\ \prettyref{lem:Descent}),
or for any global bound on it (e.g.\ \prettyref{thm:Mixing-Lemma}),
we shall need the stronger assumption for later applications.

Finally, we remark that it is sometimes useful to consider the Laplacian
$\Delta_{-1}^{+}$ as well. This operator acts on $\Omega^{-1}\cong\mathbb{R}$
as multiplication by $n=\left|V\right|$, so that every complex is
automatically a $\left(-1,n,0\right)$-expander.

\section{\label{sec:The-main-theorems}The main theorems}

In this section we assume that $X$ is a $d$-complex on $n$ vertices,
which is a $\left(j,k_{j},\varepsilon_{j}\right)$-expander for $0\leq j<d$,
and prove the Descent Lemma (\prettyref{lem:Descent}) and the mixing
lemmas it implies.
\begin{proof}[Proof of the Descent Lemma]
To any disjoint sets of vertices $A_{0},\ldots,A_{j}$, we associate
the characteristic $j$-form $\delta_{A_{0}\ldots A_{j}}\in\Omega^{j}$,
which takes $\pm1$ on $j$-cells in $F\left(A_{0},\ldots,A_{j}\right)$
(according to their orientation), and vanishes elsewhere: 
\[
\delta_{A_{0}\ldots A_{j}}\left(\sigma\right)=\begin{cases}
\sgn\left(\pi\right) & \exists\pi\in\mathrm{Sym}_{\left\{ 0\ldots j\right\} }\:\mathrm{with}\:\sigma_{i}\in A_{\pi\left(i\right)}\:\mathrm{for}\:0\leq i\leq j\\
0 & \mathrm{otherwise.}
\end{cases}
\]
Multiplication by $\delta_{A_{0}\ldots A_{j}}$ forms a projection
operator on $\Omega^{j}$, which we denote by $\mathbb{P}_{A_{0}\ldots A_{j}}$:
\[
\mathbb{P}_{A_{0}\ldots A_{j}}\left(\varphi\right)=\delta_{A_{0}\ldots A_{j}}\cdot\varphi=\sigma\mapsto\delta_{A_{0}\ldots A_{j}}\left(\sigma\right)\varphi\left(\sigma\right).
\]
We start our analysis by observing for disjoint sets $A_{0},\ldots,A_{j+1}$
the form $\left(-1\right)^{j}\mathbb{P}_{A_{0}\ldots A_{j}}\mathcal{A}_{j}^{\sim}\delta_{A_{1}\ldots A_{j+1}}$
vanishes outside $F\left(A_{0},\ldots,A_{j}\right)$, and to each
$j$-cell therein it assigns the number of its $\sim$-neighbors in
$F\left(A_{1},\ldots,A_{j+1}\right)$. As these neighbors are in correspondence
with $\left(j+1\right)$-cells in $F\left(A_{0},\ldots,A_{j+1}\right)$,
we obtain that $\left|\left\langle \delta_{A_{0}\ldots A_{j}},\mathbb{P}_{A_{0}\ldots A_{j}}\mathcal{A}_{j}^{\sim}\delta_{A_{1}\ldots A_{j+1}}\right\rangle \right|=\left|F\left(A_{0},\ldots,A_{j+1}\right)\right|$. 

Next, let $\varphi$ be a $j$-form which is supported on $F\left(A_{1},\ldots,A_{j+1}\right)$,
and which assigns to each $j$-cell $\sigma$ the number of $\left(j+1\right)$-galleries
in $A_{1},\ldots,A_{\ell}$ whose first cell contains $\sigma$. By
the same consideration as above, $\left(-1\right)^{j}\mathbb{P}_{A_{0}\ldots A_{j}}\mathcal{A}_{j}^{\sim}\varphi$
assigns to every $j$-cell $\tau$ in $F\left(A_{0},\ldots,A_{j}\right)$
the number of $\left(j+1\right)$-galleries in $A_{0},\ldots,A_{\ell}$
whose first $\left(j+1\right)$ cell contains $\tau$. Therefore,
$\left|\left\langle \delta_{A_{0}\ldots A_{j}},\mathbb{P}_{A_{0}\ldots A_{j}}\mathcal{A}_{j}^{\sim}\varphi\right\rangle \right|=\left|F^{j+1}\left(A_{0},\ldots,A_{\ell}\right)\right|$,
and we conclude by induction that 
\begin{equation}
\left|F^{j+1}\left(A_{0},\ldots,A_{\ell}\right)\right|=\left|\left\langle \delta_{A_{0}\ldots A_{j}},\left(\prod_{i=0}^{\ell-j-1}\mathbb{P}_{A_{i}\ldots A_{i+j}}\mathcal{A}_{j}^{\sim}\right)\delta_{A_{\ell-j}\ldots A_{\ell}}\right\rangle \right|.\label{eq:F_j+1_by_upper_adj}
\end{equation}
Since the $A_{i}$ are disjoint, $\delta_{A_{i}\ldots A_{i+j}}$ and
$\delta_{A_{i+1}\ldots A_{i+j+1}}$ are supported on different cells,
so that $\mathbb{P}_{A_{i}\ldots A_{i+j}}T\delta_{A_{i+1}\ldots A_{i+j+1}}=0$
for any diagonal operator $T$. Thus, all the $\mathcal{A}_{j}^{\sim}$
in \prettyref{eq:F_j+1_by_upper_adj} can be replaced by $\mathcal{A}_{j}^{\sim}+T$,
and taking $T=k_{j}I-D_{j}$ we obtain 
\begin{equation}
\left|F^{j+1}\left(A_{0},\ldots,A_{\ell}\right)\right|=\left|\left\langle \delta_{A_{0}\ldots A_{j}},\left(\prod_{i=0}^{\ell-j-1}\mathbb{P}_{A_{i}\ldots A_{i+j}}\left(k_{j}I-\Delta_{j}^{+}\right)\right)\delta_{A_{\ell-j}\ldots A_{\ell}}\right\rangle \right|.\label{eq:F_j+1_by_upper_lap}
\end{equation}
Our next step is to approximate this quantity using the lower $j$-th
Laplacian. Denoting $E=k_{j}I-\Delta_{j}^{+}-\frac{k_{j}}{k_{j-1}}\Delta_{j}^{-}$,
the orthogonal decomposition $\Omega^{j}=Z_{j}\oplus B^{j}$ gives
\[
E=k_{j}\left(\mathbb{P}_{Z_{j}}+\mathbb{P}_{B^{j}}\right)-\Delta_{j}^{+}-\frac{k_{j}}{k_{j-1}}\Delta_{j}^{-}=k_{j}\mathbb{P}_{Z_{j}}-\Delta_{j}^{+}+\frac{k_{j}}{k_{j-1}}\left(k_{j-1}\mathbb{P}_{B^{j}}-\Delta_{j}^{-}\right).
\]
We first observe that $\left\Vert k_{j}\mathbb{P}_{Z_{j}}-\Delta_{j}^{+}\right\Vert \leq k_{j}\varepsilon_{j}$
follows from $\Spec\Delta_{j}^{+}\big|_{Z_{j}}\subseteq\left[k_{j}\left(1-\varepsilon_{j}\right),k_{j}\left(1+\varepsilon_{j}\right)\right]$
and $\Delta_{j}^{+}\big|_{B^{j}}\equiv0$. For the lower Laplacian,
we have 
\begin{multline*}
\Spec\Delta_{j}^{-}\big|_{B^{j}}=\Spec\Delta_{j}^{-}\big|_{Z_{j}^{\bot}}=\Spec\Delta_{j}^{-}\backslash\left\{ 0\right\} \overset{\left(*\right)}{=}\Spec\Delta_{j-1}^{+}\backslash\left\{ 0\right\} =\Spec\Delta_{j-1}^{+}\big|_{\left(Z^{j-1}\right)^{\bot}}\\
=\Spec\Delta_{j-1}^{+}\big|_{B_{j-1}}\subseteq\Spec\Delta_{j-1}^{+}\big|_{Z_{j-1}}\subseteq\left[k_{j-1}\left(1-\varepsilon_{j-1}\right),k_{j-1}\left(1+\varepsilon_{j-1}\right)\right],
\end{multline*}
where $\left(*\right)$ follows from the fact that $\Delta_{j}^{-}=\partial_{j}^{*}\partial_{j}$
and $\Delta_{j-1}^{+}=\partial_{j}\partial_{j}^{*}$. As $\Delta_{j}^{-}$
vanishes on $Z_{j}$, we have in total $\left\Vert k_{j-1}\mathbb{P}_{B^{j}}-\Delta_{j}^{-}\right\Vert \leq k_{j-1}\varepsilon_{j-1}$,
so that
\begin{equation}
\left\Vert E\right\Vert \leq\left\Vert k_{j}\mathbb{P}_{Z_{j}}-\Delta_{j}^{+}\right\Vert +\frac{k_{j}}{k_{j-1}}\left\Vert k_{j-1}\mathbb{P}_{B^{j}}-\Delta_{j}^{-}\right\Vert \leq k_{j}\left(\varepsilon_{j-1}+\varepsilon_{j}\right).\label{eq:E-bound}
\end{equation}
We proceed to expand \prettyref{eq:F_j+1_by_upper_lap}, using $\frac{k_{j}}{k_{j-1}}\Delta_{j}^{-}+E=k_{j}I-\Delta_{j}^{+}$,
and on occasions translating $\Delta_{j}^{-}$ by some diagonal (in
fact, scalar) operators: 
\begin{align}
\left|F^{j+1}\left(A_{0},\ldots,A_{\ell}\right)\right| & =\left|\left\langle \delta_{A_{0}\ldots A_{j}},\left(\prod_{i=0}^{\ell-j-1}\mathbb{P}_{A_{i}\ldots A_{i+j}}\left(\frac{k_{j}}{k_{j-1}}\Delta_{j}^{-}+E\right)\right)\delta_{A_{\ell-j}\ldots A_{\ell}}\right\rangle \right|\nonumber \\
 & \negthickspace\negthickspace\negthickspace\negthickspace\negthickspace\negthickspace\negthickspace\negthickspace\negthickspace\negthickspace\negthickspace\negthickspace\negthickspace\negthickspace\negthickspace\negthickspace\negthickspace\negthickspace\negthickspace\negthickspace=\left|\left(\frac{k_{j}}{k_{j-1}}\right)^{\ell-j}\left\langle \delta_{A_{0}\ldots A_{j}},\left(\prod_{i=0}^{\ell-j-1}\mathbb{P}_{A_{i}\ldots A_{i+j}}\Delta_{j}^{-}\right)\delta_{A_{\ell-j}\ldots A_{\ell}}\right\rangle \right.\nonumber \\
 & \negthickspace\negthickspace\negthickspace\negthickspace\negthickspace\negthickspace\negthickspace\negthickspace\negthickspace\negthickspace\negthickspace\negthickspace\negthickspace\negthickspace\negthickspace\negthickspace\left.+\sum_{m=1}^{\ell-j}\left(\frac{k_{j}}{k_{j-1}}\right)^{\ell-j-m}\left\langle \delta_{A_{0}\ldots A_{j}},{\left(\prod\limits _{i=0}^{\ell-j-m-1}\mathbb{P}_{A_{i}\ldots A_{i+j}}\Delta_{j}^{-}\right)\mathbb{P}_{A_{\ell-j-m}\ldots A_{\ell-m}}E\cdot\qquad\qquad\atop \quad\cdot\left(\prod\limits _{i=\ell-j-m+1}^{\ell-j-1}\mathbb{P}_{A_{i}\ldots A_{i+j}}\left(\frac{k_{j}}{k_{j-1}}\Delta_{j}^{-}+E\right)\right)\delta_{A_{\ell-j}\ldots A_{\ell}}}\right\rangle \right|\nonumber \\
 & \negthickspace\negthickspace\negthickspace\negthickspace\negthickspace\negthickspace\negthickspace\negthickspace\negthickspace\negthickspace\negthickspace\negthickspace\negthickspace\negthickspace\negthickspace\negthickspace\negthickspace\negthickspace\negthickspace\negthickspace=\left|\left(\frac{k_{j}}{k_{j-1}}\right)^{\ell-j}\left\langle \delta_{A_{0}\ldots A_{j}},\left(\prod_{i=0}^{\ell-j-1}\mathbb{P}_{A_{i}\ldots A_{i+j}}\mathcal{A}_{j}^{\pitchfork}\right)\delta_{A_{\ell-j}\ldots A_{\ell}}\right\rangle \right.\label{eq:expectancy}\\
 & \negthickspace\negthickspace\negthickspace\negthickspace\negthickspace\negthickspace\negthickspace\negthickspace\negthickspace\negthickspace\negthickspace\negthickspace\negthickspace\negthickspace\negthickspace\negthickspace\left.+\sum_{m=1}^{\ell-j}\left(\frac{k_{j}}{k_{j-1}}\right)^{\ell-j-m}\left\langle \delta_{A_{0}\ldots A_{j}},{\left(\prod\limits _{i=0}^{\ell-j-m-1}\mathbb{P}_{A_{i}\ldots A_{i+j}}\left(\Delta_{j}^{-}-k_{j-1}I\right)\right)\mathbb{P}_{A_{\ell-j-m}\ldots A_{\ell-m}}E\cdot\atop \qquad\qquad\cdot\left(\prod\limits _{i=\ell-j-m+1}^{\ell-j-1}\mathbb{P}_{A_{i}\ldots A_{i+j}}\left(k_{j}I-\Delta_{j}^{+}\right)\right)\delta_{A_{\ell-j}\ldots A_{\ell}}}\right\rangle \right|.\nonumber 
\end{align}
We first study the summand in line \prettyref{eq:expectancy}. Note
that the form $\left(-1\right)^{j}\mathbb{P}_{A_{0}\ldots A_{j}}\mathcal{A}_{j}^{\pitchfork}\delta_{A_{1}\ldots A_{j+1}}$
assigns to every $j$-cell in $F\left(A_{0},\ldots,A_{j}\right)$
the number of $j$-cells in $F\left(A_{1},\ldots,A_{j+1}\right)$
with which it intersects, so that $\left|\left\langle \delta_{A_{0}\ldots A_{j}},\mathbb{P}_{A_{0}\ldots A_{j}}\mathcal{A}_{j}^{\pitchfork}\delta_{A_{1}\ldots A_{j+1}}\right\rangle \right|=\left|F^{j}\left(A_{0},\ldots,A_{j+1}\right)\right|$
(recall that for $A_{j}^{\sim}$ in place of $A_{j}^{\pitchfork}$
we obtained $\left|F^{j+1}\left(A_{0},\ldots,A_{j+1}\right)\right|$).
By the same arguments as before one sees that
\[
\left|F^{j}\left(A_{0},\ldots,A_{\ell}\right)\right|=\left|\left\langle \delta_{A_{0}\ldots A_{j}},\left(\prod_{i=0}^{\ell-j-1}\mathbb{P}_{A_{i}\ldots A_{i+j}}\mathcal{A}_{j}^{\pitchfork}\right)\delta_{A_{\ell-j}\ldots A_{\ell}}\right\rangle \right|,
\]
so that line \prettyref{eq:expectancy} is precisely $\left(\frac{k_{j}}{k_{j-1}}\right)^{\ell-j}\left|F^{j}\left(A_{0},\ldots,A_{\ell}\right)\right|$,
our estimate for $\left|F^{j+1}\left(A_{0},\ldots,A_{\ell}\right)\right|$.
Denoting by $\mathcal{E}$ the error term (the line below \prettyref{eq:expectancy}),
we bound it using \prettyref{eq:E-bound} together with $\left\Vert \Delta_{j}^{-}-k_{j-1}I\right\Vert \leq k_{j-1}$
and $\left\Vert k_{j}I-\Delta_{j}^{+}\right\Vert \leq k_{j}$ (both
follow from the discussion preceding \prettyref{eq:E-bound}): 
\begin{align*}
\mathcal{E} & \leq\sum_{m=1}^{\ell-j}\left(\frac{k_{j}}{k_{j-1}}\right)^{\ell-j-m}\left\Vert \delta_{A_{0}\ldots A_{j}}\right\Vert k_{j-1}^{\ell-j-m}k_{j}\left(\varepsilon_{j-1}+\varepsilon_{j}\right)k_{j}^{m-1}\left\Vert \delta_{A_{\ell-j}\ldots A_{\ell}}\right\Vert \\
 & =\left(\ell-j\right)k_{j}^{\ell-j}\left(\varepsilon_{j-1}+\varepsilon_{j}\right)\sqrt{\left|F\left(A_{0},\ldots,A_{j}\right)\right|\left|F\left(A_{\ell-j},\ldots,A_{\ell}\right)\right|},
\end{align*}
which concludes the proof. 

We remark that a slightly better bound is possible here: As $\Spec\Delta_{j}^{+}\subseteq\left[0,k_{j}\left(1+\varepsilon_{j}\right)\right]$,
we can replace $k_{j}I-\Delta_{j}^{+}$ in the line below \prettyref{eq:expectancy}
by $\frac{k_{j}\left(1+\varepsilon_{j}\right)}{2}I-\Delta_{j}^{+}$,
which is bounded by $\frac{k_{j}\left(1+\varepsilon_{j}\right)}{2}$,
and likewise for $\Delta_{j}^{-}$ (whose spectrum lies within $\left[0,k_{j-1}\left(1+\varepsilon_{j-1}\right)\right]$).
For example, putting $\varepsilon=\max\varepsilon_{i}$ this gives
\[
\mathcal{E}\leq\left(\ell-j\right)k_{j}^{\ell-j}2\varepsilon\left(\frac{1+\varepsilon}{2}\right)^{\ell-j-1}\sqrt{\left|F\left(A_{0},\ldots,A_{j}\right)\right|\left|F\left(A_{\ell-j},\ldots,A_{\ell}\right)\right|}
\]
which might be useful when all $\varepsilon_{i}$ are small. 
\end{proof}
Using the Descent Lemma repeatedly gives:
\begin{prop}
\label{prop:from_j_to_l}For any $j<\ell$, there exists $c_{j,\ell}$
such that any disjoint sets of vertices $A_{0},\ldots,A_{\ell}$ in
a $\left(\overline{k},\overline{\varepsilon}\right)$-expander satisfy
\[
\left|\left|F^{j+1}\left(A_{0},\ldots,A_{\ell}\right)\right|-\frac{k_{0}k_{1}\ldots k_{j-1}k_{j}^{\ell-j}}{n^{\ell}}\prod_{i=0}^{\ell}\left|A_{i}\right|\right|\leq c_{j,\ell}k_{0}k_{1}\ldots k_{j-1}k_{j}^{\ell-j}\left(\varepsilon_{0}+\ldots+\varepsilon_{j}\right)\max\left|A_{i}\right|.
\]

\end{prop}
In particular, for $j=d-1$, $\ell=d$ we obtain \prettyref{thm:Mixing-Lemma}:
\begin{thm*}[\prettyref{thm:Mixing-Lemma}]
Any disjoint sets of vertices $A_{0},\ldots,A_{d}$ in a $\left(\overline{k},\overline{\varepsilon}\right)$-expander
of dimension $d$ satisfy
\[
\left|\left|F\left(A_{0},\ldots,A_{d}\right)\right|-\frac{k_{0}\ldots k_{d-1}}{n^{d}}\left|A_{0}\right|\cdot\ldots\cdot\left|A_{d}\right|\right|\leq c_{d}k_{0}\ldots k_{d-1}\left(\varepsilon_{0}+\ldots+\varepsilon_{d-1}\right)\max\left|A_{i}\right|,
\]
for some constant $c_{d}$ which depends only on $d$.\end{thm*}
\begin{proof}[Proof of \prettyref{prop:from_j_to_l}]
We denote $m=\max\left|A_{i}\right|$ and assume by induction that
the proposition holds for $j-1$ (and any $\ell$), i.e.\ that 
\begin{equation}
\left|F^{j}\left(A_{0},\ldots,A_{\ell}\right)-\frac{k_{0}k_{1}\ldots k_{j-2}k_{j-1}^{\ell-j+1}}{n^{\ell}}\prod_{i=0}^{\ell}\left|A_{i}\right|\right|\leq c_{j-1,\ell}mk_{0}k_{1}\ldots k_{j-2}k_{j-1}^{\ell-j+1}\left(\varepsilon_{0}+\ldots+\varepsilon_{j-1}\right).\label{eq:induc-j}
\end{equation}
For $j=0$ this indeed holds, in the sense that
\begin{equation}
\left|F^{0}\left(A_{0},\ldots,A_{\ell}\right)-\frac{k_{-1}^{\ell}}{n^{\ell}}\prod_{i=0}^{\ell}\left|A_{i}\right|\right|=0.\label{eq:induc-base}
\end{equation}
Let us denote by $\mathcal{E}$ the discrepancy $\left|\left|F^{j+1}\left(A_{0},\ldots,A_{\ell}\right)\right|-\frac{k_{0}k_{1}\ldots k_{j-1}k_{j}^{\ell-j}}{n^{\ell}}\prod_{i=0}^{\ell}\left|A_{i}\right|\right|$.
Combining the Descent Lemma with \prettyref{eq:induc-j} (or \prettyref{eq:induc-base},
for $j=0$) multiplied by $\left(\frac{k_{j}}{k_{j-1}}\right)^{\ell-j}$
gives 
\begin{align*}
\mathcal{E} & \leq\left(\ell-j\right)k_{j}^{\ell-j}\left(\varepsilon_{j}+\varepsilon_{j-1}\right)\sqrt{\left|F\left(A_{0},\ldots,A_{j}\right)\right|\left|F\left(A_{\ell-j},\ldots,A_{\ell}\right)\right|}\\
 & \phantom{\leq}+c_{j-1,\ell}mk_{0}k_{1}\ldots k_{j-1}k_{j}^{\ell-j}\left(\varepsilon_{0}+\ldots+\varepsilon_{j-1}\right).
\end{align*}
To bound $\left|F\left(A_{0},\ldots,A_{j}\right)\right|=\left|F^{j}\left(A_{0},\ldots,A_{j}\right)\right|$
we use \prettyref{eq:induc-j} with $\ell=j$, which gives
\begin{align*}
\left|F^{j}\left(A_{0},\ldots,A_{j}\right)\right| & \leq\frac{k_{0}\ldots k_{j-1}}{n^{j}}\prod_{i=0}^{j}\left|A_{i}\right|+c_{j-1,j}mk_{0}\ldots k_{j-1}\left(\varepsilon_{0}+\ldots+\varepsilon_{j-1}\right)\\
 & \leq\left[1+c_{j-1,j}\left(\varepsilon_{0}+\ldots+\varepsilon_{j-1}\right)\right]mk_{0}\ldots k_{j-1}\\
 & \leq\left(1+jc_{j-1,j}\right)mk_{0}\ldots k_{j-1}.
\end{align*}
(here we used $\varepsilon_{i}<1$, but any bound on the $\varepsilon_{i}$
would do). The same holds for $\left|F\left(A_{\ell-j},\ldots,A_{\ell}\right)\right|$,
hence
\begin{align*}
\mathcal{E} & \leq\left(\ell-j\right)k_{j}^{\ell-j}\left(\varepsilon_{j}+\varepsilon_{j-1}\right)\left(1+jc_{j-1,j}\right)mk_{0}\ldots k_{j-1}\\
 & \phantom{\leq}+c_{j-1,\ell}mk_{0}k_{1}\ldots k_{j-1}k_{j}^{\ell-j}\left(\varepsilon_{0}+\ldots+\varepsilon_{j-1}\right)\\
 & =mk_{0}k_{1}\ldots k_{j-1}k_{j}^{\ell-j}\left[c_{j-1,\ell}\left(\varepsilon_{0}+\ldots+\varepsilon_{j-1}\right)+\left(\ell-j\right)\left(1+jc_{j-1,j}\right)\left(\varepsilon_{j}+\varepsilon_{j-1}\right)\right]\\
 & \leq\underbrace{\left[c_{j-1,\ell}+\left(\ell-j\right)\left(1+jc_{j-1,j}\right)\right]}_{c_{j,\ell}}mk_{0}k_{1}\ldots k_{j-1}k_{j}^{\ell-j}\left(\varepsilon_{0}+\ldots+\varepsilon_{j}\right).
\end{align*}
as desired.
\end{proof}

\section{\label{sec:Applications}Applications}

The following notion of geometric expansion for graphs and complexes
originates in Gromov's work \cite{Gro10} (see also \cite{FGL+11,MW11}):
\begin{defn}
Let $X$ be a $d$-dimensional simplicial complex. The \emph{geometric
overlap }of $X$ is 
\[
\overlap X=\min_{\varphi:V\rightarrow\mathbb{R}^{d}}\,\max_{x\in\mathbb{R}^{d}}\,\frac{\#\left\{ \sigma\in X^{d}\,\middle|\, x\in\mathrm{conv}\left\{ \varphi\left(v\right)\,\middle|\, v\in\sigma\right\} \right\} }{\left|X^{d}\right|}.
\]
In other words, $X$ has $\overlap\geq\varepsilon$ if for every simplicial
mapping of $X$ into $\mathbb{R}^{d}$ (a mapping induced linearly
by the images of the vertices), some point in $\mathbb{R}^{d}$ is
covered by at least an $\varepsilon$-fraction of the $d$-cells of
$X$. 
\end{defn}
A theorem of Pach \cite{Pac98} relates combinatorial expansion and
geometric overlap, and allows us to prove the following:
\begin{prop}
\label{prop:geom-overlap}If $X$ is a $d$-dimensional $\left(\overline{k},\overline{\varepsilon}\right)$-expander
then
\[
\overlap X>\frac{\mathcal{P}_{d}d!}{2^{d}}\left[\left(\frac{\mathcal{P}_{d}}{d+1}\right)^{d}-c_{d}\left(\varepsilon_{0}+\ldots+\varepsilon_{d-1}\right)\right],
\]
where $\mathcal{P}_{d}$ is Pach's constant \cite{Pac98}, and $c_{d}$
is the constant from \prettyref{thm:Mixing-Lemma} (both depend only
on $d$).
\end{prop}
In particular, a family of $d$-complexes which have $\varepsilon_{0}+\ldots+\varepsilon_{d-1}$
small enough is a family of geometric expanders. For the proof of
\prettyref{prop:geom-overlap} we shall need the following lemma,
which relates the Laplace spectrum to cell density:
\begin{lem}
\label{lem:degree-from-spec}Let $X$ be a $d$-complex with $\beta_{j}=0$
for $j<d$, and let $\lambda_{j}$ be the average nontrivial eigenvalue
of $\Delta_{j}^{+}$, for $-1\leq j<d$ (in particular $\lambda_{-1}=n$).
For any $0\leq m<d$ the average degree of an $m$-cell is
\begin{equation}
\avg\left\{ \deg\sigma\,\middle|\,\sigma\in X^{m}\right\} =\lambda_{m}\left(1-\frac{m+1}{\lambda_{m-1}}\right),\label{eq:avg-deg}
\end{equation}
and the number of $m$-cells is 
\begin{equation}
\left|X^{m}\right|=\frac{\lambda_{m-1}}{m+1}\cdot\prod_{j=-1}^{m-2}\left(\frac{\lambda_{j}}{j+2}-1\right)=\frac{\lambda_{m-1}\left(n-1\right)}{m+1}\cdot\prod_{j=0}^{m-2}\left(\frac{\lambda_{j}}{j+2}-1\right).\label{eq:num-d-cells}
\end{equation}
\end{lem}
\begin{proof}
Since the trivial spectrum of $\Delta_{j}^{+}$ consists of zeros,
\[
\left|X^{m}\right|=\frac{1}{m+1}\sum_{\sigma\in X^{m-1}}\deg\sigma=\frac{1}{m+1}\tr D_{m-1}=\frac{1}{m+1}\tr\Delta_{m-1}^{+}=\frac{\lambda_{m-1}}{m+1}\dim Z_{m-1}.
\]
Thus, \prettyref{eq:num-d-cells} is equivalent to the assertion that
\[
\dim Z_{m-1}=\prod_{j=-1}^{m-2}\left(\frac{\lambda_{j}}{j+2}-1\right).
\]
This is true for $m=0$, and by induction, together with the triviality
of the $\left(m-2\right)$-th homology we find that
\begin{gather*}
\dim Z_{m-1}=\dim\Omega^{m-1}-\dim B_{m-2}=\left|X^{m-1}\right|-\dim Z_{m-2}\\
=\frac{\lambda_{m-2}}{m}\prod_{j=-1}^{m-3}\left(\frac{\lambda_{j}}{j+2}-1\right)-\prod_{j=-1}^{m-3}\left(\frac{\lambda_{j}}{j+2}-1\right)=\prod_{j=-1}^{m-2}\left(\frac{\lambda_{j}}{j+2}-1\right)
\end{gather*}
as desired. Formula \prettyref{eq:avg-deg} follows from \prettyref{eq:num-d-cells},
as $\avg\left\{ \deg\sigma\,\middle|\,\sigma\in X^{m}\right\} =\left(m+2\right)\left|X^{m+1}\right|/\left|X^{m}\right|$.
\end{proof}
We can now proceed:
\begin{proof}[Proof of \prettyref{prop:geom-overlap}]
Let $\varphi$ be a simplicial map $X\rightarrow\mathbb{R}^{d}$,
and divide $V=X^{0}$ arbitrarily into parts $P_{0},\ldots,P_{d+1}$
of equal size $\left|P_{i}\right|=\frac{n}{d+1}$. Pach's theorem
then states that there exist $Q_{i}\subseteq P_{i}$ of size $\left|Q_{i}\right|=\mathcal{P}_{d}\left|P_{i}\right|$
and a point $x\in\mathbb{R}^{d+1}$, such that $x\in\mathrm{conv}\left\{ \varphi\left(v\right)\,\middle|\, v\in\sigma\right\} $
for all $\sigma\in F\left(Q_{0},\ldots,Q_{d}\right)$. Denoting $\mathcal{K}=k_{0}\cdot\ldots\cdot k_{d-1}$
and $\mathcal{E}=\varepsilon_{0}+\ldots+\varepsilon_{d-1}$, we have
by \prettyref{thm:Mixing-Lemma}
\[
\left|F\left(Q_{0},\ldots,Q_{d}\right)\right|\geq\frac{\mathcal{K}}{n^{d}}\left(\frac{\mathcal{P}_{d}n}{d+1}\right)^{d+1}-\frac{c_{d}\mathcal{P}_{d}n\mathcal{K}\mathcal{E}}{d+1}=\frac{\mathcal{K}\mathcal{P}_{d}n}{d+1}\left[\left(\frac{\mathcal{P}_{d}}{d+1}\right)^{d}-c_{d}\mathcal{E}\right],
\]
and by the lemma above
\[
\left|X^{d}\right|=\frac{\lambda_{d-1}}{d+1}\cdot\prod_{j=-1}^{d-2}\left(\frac{\lambda_{j}}{j+2}-1\right)\leq\prod_{j=-1}^{d-1}\frac{\lambda_{j}}{j+2}\leq n\prod_{j=0}^{d-1}\frac{k_{j}\left(1+\varepsilon_{j}\right)}{j+2}<\frac{2^{d}n\mathcal{K}}{\left(d+1\right)!}.
\]
This means $x$ is covered by at least a $\frac{\mathcal{P}_{d}d!}{2^{d}}\left(\left(\frac{\mathcal{P}_{d}}{d+1}\right)^{d}-c_{d}\mathcal{E}\right)$-fraction
of the $d$-cells, and as this is true for all $\varphi$ the proposition
follows.
\end{proof}
We turn our attention to colorings. We say that a $d$-complex $X$
is $c$-colorable if there is a coloring of its vertices by $c$ colors
so that no $d$-cell is monochromatic. The \emph{chromatic number}
of $X$, denoted $\chi\left(X\right)$, is the smallest $c$ for which
$X$ is $c$-colorable. We will use the mixing property to show that
spectral expansion implies a chromatic bound, as is done for graphs
in \cite{LPS88}. These results are weaker than Hoffman's chromatic
bound for graphs \cite{hoffman1970eigenvalues}, as they require a
two-sided spectral bound, and the chromatic bound obtained is not
optimal. A chromatic bound for complexes which does generalize Hoffman's
result was recently obtained in \cite{golubev2013chromatic}.
\begin{prop}
If $X$ is a $d$-dimensional $\left(\overline{k},\overline{\varepsilon}\right)$-expander,
then
\[
\chi\left(X\right)\geq\frac{1}{\left(d+1\right)\sqrt[d]{c_{d}\left(\varepsilon_{0}+\ldots+\varepsilon_{d-1}\right)}},
\]
where $c_{d}$ is the constant from \prettyref{thm:Mixing-Lemma}.\end{prop}
\begin{proof}
Coloring $X$ by $\chi=\chi\left(X\right)$ colors, there is necessarily
a monochromatic set of vertices of size at least $\frac{n}{\chi}$.
Take $\frac{n}{\chi}$ of these vertices and partition them arbitrarily
to $d+1$ sets $A_{0},\ldots,A_{d}$ of equal size. As in a coloring
there are no monochromatic $d$-cells we have $F\left(A_{0},\ldots,A_{d}\right)=\varnothing$,
so that \prettyref{thm:Mixing-Lemma} reads
\[
\frac{k_{0}\ldots k_{d-1}}{n^{d}}\prod_{i=0}^{d}\left|A_{i}\right|\leq c_{d}k_{0}\ldots k_{d-1}\left(\varepsilon_{0}+\ldots+\varepsilon_{d-1}\right)\max\left|A_{i}\right|,
\]
and since  $\left|A_{i}\right|=\frac{n}{\chi\cdot\left(d+1\right)}$,
the conclusion follows.
\end{proof}

\subsection{Ideal expanders}

Let us say that $X$ is an \emph{ideal $\overline{k}$-expander} if
it is a $\left(j,k_{j},0\right)$-expander for $0\leq j<d$. In this
case, the Descent Lemma tell us that 
\[
F^{j+1}\left(A_{0},\ldots,A_{\ell}\right)=\left(\frac{k_{j}}{k_{j-1}}\right)^{\ell-j}\left|F^{j}\left(A_{0},\ldots,A_{\ell}\right)\right|,
\]
and the number of $j$-galleries between disjoint sets of vertices
is completely determined by their sizes: 
\begin{equation}
\left|F^{j}\left(A_{0},\ldots,A_{\ell}\right)\right|=\frac{k_{0}k_{1}\ldots k_{j-2}k_{j-1}^{\ell-j+1}}{n^{\ell}}\prod_{i=0}^{\ell}\left|A_{i}\right|\label{eq:ideal-mixing}
\end{equation}
(in particular, $\left|F\left(A_{0},\ldots,A_{d}\right)\right|=\frac{k_{0}\ldots k_{d-1}}{n^{d}}\left|A_{0}\right|\ldots\left|A_{d}\right|$).
For $k_{j}=\begin{cases}
n & \,0\,\leq j<m\\
0 & m\leq j<d
\end{cases}$, an example of an ideal $\overline{k}$-expander is given by $K_{n}^{\left(m\right)}$,
the $m$-th skeleton of the complete complex on $n$ vertices. For
this complex \prettyref{eq:ideal-mixing} holds trivially, and perhaps
disappointingly, these are the only examples of ideal expanders: if
$X$ is an ideal $\overline{k}$-expander on $n$ vertices, and $X^{\left(j\right)}=K_{n}^{\left(j\right)}$
(which holds for $j=0$), one has $k_{0}=\ldots=k_{j-1}=n$, and also
$k_{j}\leq n$ by \cite[prop.\ 3.2(2)]{parzanchevski2012isoperimetric}.
For vertices $v_{0},\ldots,v_{j+1}$, $\left|F\left(\left\{ v_{0}\right\} ,\ldots,\left\{ v_{j+1}\right\} \right)\right|=\frac{k_{0}\ldots k_{j}}{n^{j+1}}\in\left\{ 0,1\right\} $
then forces either $k_{j}=n$, which implies that $X^{\left(j+1\right)}=K_{n}^{\left(j+1\right)}$
as well, or $k_{j}=0$, which means that $X$ has no $\left(j+1\right)$-cells
at all.

While ideal $\overline{k}$-expanders do not actually exist, save
for the trivial examples $\overline{k}=\left(n,\ldots,n,0,\ldots\right)$,
they provide a conceptual way to think of expanders in general: $\left(\overline{k},\overline{\varepsilon}\right)$-expanders
spectrally approximate the ideal (nonexistent) $\overline{k}$-expander,
and the mixing lemma asserts that they also combinatorially approximate
it. This point of view seems close in spirit to that of \emph{spectral
sparsification} \cite{spielman2011spectral}, which proved to be fruitful
in both graphs and complexity theory.

\section{\label{sec:Questions}Questions}

Several natural questions arise from this study:

$\bullet$\ \ In \cite{Gundert2012} it is shown that random complexes
in the Linial-Meshulam model \cite{LM06} have spectral concentration
for appropriate parameters (see also \cite[§4.5]{parzanchevski2012isoperimetric}).
These are complexes with a complete skeleton, which are high-dimensional
analogues of Erd\H{o}s\textendash{}Rényi graphs. Is there a similar
model for general complexes, for which the skeletons are not complete
(preferably, where the expected degrees of cells are only logarithmic
in the number of vertices), with concentrated spectrum?

$\bullet$\ \ A well known source of excellent expanders are random
regular graphs (see, e.g.\ \cite{puder2012expansion,Friedman2008}).
Can one construct a model for random regular complexes, and are these
complexes high-dimensional expanders? This is interesting even for
a weak notion of regularity, such as having a bounded fluctuation
of degrees, or having all links of vertices isomorphic.

$\bullet$\ \ Ramanujan graphs, constructed in \cite{LPS88,margulis1988explicit,marcus2013interlacing}
are another source of optimal expanders. Ramanujan complexes, their
higher dimensional counterparts, were defined and studied in \cite{cartwright2003ramanujan,li2004ramanujan,Lubotzky2005a},
but as yet not from the point of view of the Hodge Laplacian. It is
natural to conjecture that they form spectral, and thus combinatorial
expanders, as in the case of graphs.

$\bullet$\ \ In \cite{parzanchevski2012isoperimetric} a generalization
for the discrete Cheeger inequality is given, for complexes with a
complete skeleton. Can this result be generalized to arbitrary complexes?

$\bullet$\ \ Can one prove a converse to the Expander Mixing Lemma
in general dimension, as is done for graphs in \cite{BL06}?

\bibliographystyle{amsalpha}
\bibliography{/home/ori/Math/mybib}

\lyxaddress{\noun{School of Mathematics}\\
\noun{Institute for Advanced Study}\\
\noun{Princeton, NJ 08540 }\\
E-mail: \texttt{parzan@ias.edu}}

\end{document}